\documentclass[11pt,a4paper,english,reqno]{amsart}
\usepackage{amsmath,amssymb,amsfonts,epsfig,mathrsfs}
\usepackage[T1]{fontenc}

\usepackage{color}
\usepackage{array}
\usepackage{amsthm}
\usepackage{amstext}
\usepackage{graphicx}
\usepackage{setspace}
\usepackage{mathrsfs}
\usepackage[margin=2.5cm]{geometry}
\usepackage{bbm}
\usepackage{color}
\usepackage{enumitem}
\usepackage{undertilde}
\allowdisplaybreaks[4]

\usepackage{pgfplots}

\usepackage{amscd,psfrag}
\usepackage{yhmath}

\usepackage{nicefrac}

\usepackage{slashed}

\makeatletter
\pdfpageheight\paperheight
\pdfpagewidth\paperwidth

\usepackage{epstopdf}
\usepackage{chngcntr}
\counterwithin{figure}{section}
\usepackage{mathrsfs}

\usepackage{indentfirst}	

\usepackage[normalem]{ulem}
\theoremstyle{plain}

\newtheorem{definition}{Definition}[section]
\newtheorem{theorem}[definition]{Theorem}
\newtheorem*{theorem*}{Theorem}

\newtheorem{remark}[definition]{Remark}

\newtheorem*{remark*}{Remark}
\newtheorem*{sideremark*}{Side Remark}\newtheorem*{mt*}{Main Theorem}

\newtheorem*{claim*}{Claim}
\newtheorem*{q*}{Question}
\newtheorem{lemma}[definition]{Lemma}

\newtheorem*{corollary*}{Corollary}
\newtheorem*{proposition*}{Proposition}

\newcommand{\R}{\mathbb{R}}

\newcommand{\na}{\nabla}

\newcommand{\dd}{{\rm d}}
\newcommand{\p}{\partial}
\newcommand{\e}{\epsilon}

\newcommand{\map}{\rightarrow}
\newcommand{\G}{\Gamma}
\newcommand{\M}{\mathcal{M}}

\newcommand{\hau}{{\mathscr{H}}}
\newcommand{\pr}{{\bf pr}}
\newcommand{\XX}{{\mathscr{X}}}\newcommand{\YY}{{\mathscr{Y}}}\newcommand{\ZZ}{{\mathscr{Z}}}\newcommand{\TT}{{\mathscr{T}}}
\newcommand{\J}{\mathcal{J}}
\newcommand{\sss}{{\mathscr{S}}}
\newcommand{\gs}{{\overline{g}}}

\newcommand{\mres}{\mathbin{\vrule height 1.6ex depth 0pt width
0.13ex\vrule height 0.13ex depth 0pt width 1.3ex}}

\allowdisplaybreaks[4]

\def\XXint#1#2#3{{\setbox0=\hbox{$#1{#2#3}{\int}$ }
\vcenter{\hbox{$#2#3$ }}\kern-.6\wd0}}

\numberwithin{equation}{section}
\numberwithin{figure}{section}

\title{A note on Alberti's Luzin-type theorem for gradients}

\author{Siran Li}
\address{Siran Li: Department of Mathematics, Rice University, MS 136
P.O. Box 1892, Houston, Texas, USA (77251)}

\email{\texttt{Siran.Li@rice.edu}}

\keywords{}
\subjclass[2010]{Primary: Luzin-type theorem; simplicial approximation; $C^2$-rectifiable set; contact geometry.}

\date{\today}

\begin{document}

\maketitle

\begin{abstract}
We give a ``soft'' proof of  Alberti's Luzin-type theorem in \cite{a} (G. Alberti, A Lusin-type theorem for gradients, \textit{J. Funct. Anal.} \textbf{100} (1991)), using elementary geometric measure theory and topology. Applications to the $C^2$-rectifiability problem are also discussed.

\end{abstract}

\section{Introduction}

This paper is devoted to a new proof of the following theorem by G. Alberti \cite{a}:
\begin{theorem}\label{thm: alberti}
Let $\Omega \subset \R^N$ be an open set with finite $N$-dimensional Lebesgue measure; $N \geq 2$. Let $v: \Omega \map \R^N$ be a Borel vectorfield. Then, for any $\e>0$, there exist an open set $A \subset \Omega$ and a function $\phi \in C^1_0(\Omega)$ such that $\hau^N(A) \leq \e \hau^N(\Omega)$ and $v=\na \phi$ on $\Omega \sim A$.
\end{theorem}

Alberti's theorem says that any Borel vectorfield is ``nearly'' --- in the sense of  Luzin \cite{l} --- the gradient of a scalar potential. It can also be interpreted as follows: any differential $1$-form is ``nearly'' exact. The latter statement readily generalises to differential forms of arbitrary degree on Riemannian manifolds; see \cite{mp}, Proposition~2.3 for the Euclidean setting. 

Various improvements and generalisations of Alberti's theorem have been studied; {\it cf.} Moonens--Pfeffer \cite{mp} for an {\it a.e.}-version of Theorem~\ref{thm: alberti} (namely, $\e=0$ therein) and extension to charges/flat cochains, Francos \cite{fr} for extension to higher-order derivatives, and David \cite{da} to metric measure spaces, as well as the  references cited therein.

Alberti's original proof of Theorem~\ref{thm: alberti}  in \cite{a} is constructive: one divides $\Omega$ into dyadic cubes, approximates $v$ by affine functions on the dyadic cubes at each level, smooths in ``transition layers'' via convolution, and iteratively corrects resulting errors at the next level. 

Here we present an alternative proof using geometric measure theory and  topology. We take a new perspective by looking at {\em graphs} of functions, rather than the functions {\it per se}, and finding approximations to the graphs by topological arguments. This approach is  motivated by, among others, the seminal works of Giaquinta--Modica--Sou\u{c}ek \cite{gms1, gms2} on harmonic maps.


\section{A Rough Approximation}

Following Alberti's original approach (\cite{a}, Lemma~7), we first establish 
\begin{lemma}\label{lemma: key}
Let $\Omega \subset \R^N$ be an open set with finite $N$-dimensional Lebesgue measure ($N\geq 2$), let $v:\Omega\map\R^N$ be a Borel vectorfield, and let $\eta,\e,\vartheta$ be arbitrary positive numbers. There exist a compact set $K \Subset \Omega$ and a function $\phi \in C^1_0(\Omega)$ such that $\hau^N(\Omega\sim K) \leq \e\hau^N(\Omega)$, $\|\phi\|_{C^0(\Omega)} \leq \vartheta$,  and  $|v-\na \phi|\leq\eta$ on $K$.
\end{lemma}

Here, as in Moonens--Pfeffer \cite{mp}, we require $\phi$ to satisfy the uniform smallness condition $\|\phi\|_{C^0(\Omega)} \leq \vartheta$ in addition to Alberti's original result. Roughly speaking, $\phi$ is a ``micro-oscillation''.

\begin{proof} The arguments are divided into five steps.

\smallskip
\noindent
{\bf Step 1. Reduction to continuous, bounded vectorfields.}
    
    Let $\kappa$ be a small positive number to be specified later. Since $v$ is Borel, there is a finite number $\Lambda=\Lambda(\kappa)$ such that $B=B(\kappa):= \{x \in \Omega: |v(x)|>\Lambda\}$ satisfies $\hau^N(B)<\kappa$. On the other hand, applying the classical Luzin theorem \cite{l}, we may find a Borel set $B' \subset \Omega$ and a continuous vectorfield $v_0: \Omega \map \R^N$ such that $\hau^N(B')<\kappa$ and $v_0={v}$ on $\Omega \sim B'$. Let us set
     \begin{equation*}
     v_1 := \begin{cases}
     v_0\qquad \text{ on } \Omega \sim B,\\
     \Lambda |v_0|^{-1}v_0\qquad \text{ on } B.
     \end{cases}
     \end{equation*}
Clearly, $v_1$ coincides with $v$ on $\Omega \sim (B \cup B')$, $|v_1| \leq \Lambda$ on $\Omega$, and $v_1$ is continuous on $\Omega$.

Suppose now that Lemma~\ref{lemma: key} is already proved for the continuous, bounded vectorfield $v_1$; that is, we have a compact set $K_1 \Subset \Omega$ and a function $\phi \in C^1_c(\Omega)$ satisfying $\hau^N(\Omega \sim K_1) \leq \nicefrac{\e \hau^N(\Omega)}{2}$ and $|v_1-\na \phi|\leq {\eta}$ on $K_1$. Then, for the compact set $K:= K_1 \sim [{\rm int}(B \cup B')]$,  we have  $\hau^N(\Omega \sim K) \leq \nicefrac{\e \hau^N(\Omega)}{2} + 2\kappa$ and $|v_1-\na \phi|\leq {\eta}$ on $K$. Thus one may conclude Lemma~\ref{lemma: key} for the Borel vectorfield $v$ by taking $\kappa := \nicefrac{\e \hau^N(\Omega)}{4}$.


\smallskip
\noindent
{\bf Step 2. Reduction to smooth, compactly supported vectorfields. }

For each sufficiently small $a>0$, we denote as usual $$\Xi_a(\bullet):= a^{-N}\Xi_1(\nicefrac{\bullet}{a})$$ where $\Xi_1$ is the standard mollifier on $\R^N$, and $$\Omega_a:=\big\{x \in \Omega: {\rm dist}(x, \R^N \sim \Omega)>a\big\}.$$  For any continuous bounded $v:\Omega \map \R^N$ we take $$v_2 := (v \chi_{\overline{\Omega_\sigma}}) \star \Xi_{\nicefrac{\sigma}{10}}  \in C^\infty_c(\Omega;\R^N)$$ with $\sigma>0$ to be  specified. In fact, ${\rm spt}(v_2) \Subset \Omega_{\nicefrac{4\sigma}{5}}$. Here, $\chi_E$ is the characteristic function of set $E$, and $\star$ is the convolution operator.  
	
	Suppose that Lemma~\ref{lemma: key} is already established for smooth, compactly supported vectorfields. Then, for a compact set $K_2 \Subset \Omega_{\nicefrac{4\sigma}{5}}$ and a function $\phi_2 \in C^1_c(\Omega_{\nicefrac{4\sigma}{5}})$, there hold  $|v_2 - \na \phi_2| \leq \eta$ on $K_2$ and $\hau^N(\Omega_{\nicefrac{4\sigma}{5}} \sim K_2) \leq \nicefrac{\e \hau^N(\Omega_{\nicefrac{4\sigma}{5}})}{2}$. Take $K\equiv K_2$ and $\phi \equiv$ extension-by-zero of $\phi_2$. Clearly $|v-\na \phi| \leq \eta$ on $K$, and for the case $\hau^N(\Omega_{\nicefrac{4\sigma}{5}}\sim K)>0$ we have
	\begin{align}\label{ratios}
	\frac{\hau^N (\Omega \sim K)}{\hau^N(\Omega)} &=  \frac{\hau^N (\Omega \sim K)}{\hau^N(\Omega_{\nicefrac{4\sigma}{5}} \sim K)}\, \frac{\hau^N(\Omega_{\nicefrac{4\sigma}{5}} \sim K)}{\hau^N(\Omega_{\nicefrac{4\sigma}{5}})}\,\frac{\hau^N(\Omega_{\nicefrac{4\sigma}{5}})}{\hau^N(\Omega)}\nonumber\\
	&\leq \frac{\e}{2} \, \frac{\hau^N (\Omega \sim K)}{\hau^N(\Omega_{\nicefrac{4\sigma}{5}} \sim K)}\,\frac{\hau^N(\Omega_{\nicefrac{4\sigma}{5}})}{\hau^N(\Omega)}.
	\end{align}
Since $\nicefrac{\hau^N(\Omega_a)}{\hau^N(\Omega)} \nearrow 1$ as $a \searrow 0$ for the open set $\Omega\subset\R^N$, with $\sigma$ sufficiently small, the last two factors in the final line in \eqref{ratios} can be chosen arbitrarily close to $1$. Thus ${\hau^N (\Omega \sim K)} \leq \e{\hau^N(\Omega)}$. The above inequality holds trivially when $\hau^N(\Omega_{\nicefrac{4\sigma}{5}}\sim K)=0$, by shrinking $\sigma>0$ if necessary. This proves Lemma~\ref{lemma: key} for the continuous bounded vectorfield $v$.

\smallskip
\noindent
{\bf Step 3. Reduction to vectorfields with rectifiable graphs.}

In view of the preceding step, from now on we may assume that $v \in C^\infty_c(\Omega_{\sigma};\R^N)$ for some fixed $\sigma>0$. For the graph
\begin{equation*}
\G := {\rm graph}_\Omega\,(v):= \big\{\big(x,v(x)\big):\,x\in\Omega\big\} \subset \Omega \times \R^N,
\end{equation*}
we have
\begin{align*}
\hau^N(\G) := \int_\Omega \sqrt{1+|\na v|^2}\,\dd\hau^N \leq M,
\end{align*}
where $M$ is a finite number  depending only on $\|v\|_{C^1(\Omega;\R^N)}$ and the $N$-dimensional Lebesgue measure of $\Omega$. 
Applying Federer's structure theorem (\cite{fe1, w}), we get the decomposition $$\Gamma = Y \sqcup Z,$$ where $Y$ is $(\hau^N, N)$-rectifiable and $Z$ is purely $N$-unrectifiable.

Denote by $\pr$ the projection of $\Omega \times \R^N$ onto the first coordinate; we {\em claim} that 
\begin{equation}\label{claim}
\hau^N\big({\pr}(Z)\big)=0.
\end{equation}
Indeed, if it were false, there would be an open ball $\mathcal{O} \subset {\pr}(Z) \subset \Omega$ with $\hau^N(\mathcal{O})>0$. By the boundedness of $\Omega$ and that $v\in C^\infty_c(\Omega;\R^N)$, the supremum of $|\na v|$ over $\Omega$ is finite; thus the following transversality result holds:
\begin{equation}\label{transversality}
\inf \Big\{{\rm dist} ( T_\xi \G, \mathcal{V}) :\, \xi \in \G,\, \mathcal{V} \in {\bf Gr}(N, 2N) \text{ is  orthogonal to the image of $\pr$}  \Big\}  \geq \e_0>0.
\end{equation} 
Throughout, ${\bf Gr}(N, 2N)$ is the Grassmannian manifold of $N$-planes in $\R^{2N}$ with the natural topology: the distance between two 
$N$-planes is the operator norm of the difference of the corresponding projections. Denote by $\mu$ the Haar measure on ${\bf Gr}(N,2N)$ with a fixed normalisation.

By \eqref{transversality}, there is a neighbourhood $\mathcal{N} \subset {\bf Gr}(N,2N)$ such that
\begin{itemize}
\item
$\mathcal{N}$ contains the horizontal section $\R^N \times \{0\}$;
\item
Each $\Pi \in \mathcal{N}$ contains the origin of $\R^{2N}$;
\item
$\mu (\mathcal{N})>0$; and
\item
$\G_{\mathcal{O}}:={\rm graph}_{\mathcal{O}}(v)$ is  a graph over each $N$-plane in $\mathcal{N}$.
\end{itemize}

Writing ${\pr}_\Pi$ for the projection of $\G_\mathcal{O}$ onto $\Pi \in \mathcal{N}$, we deduce that $\{{\pr}_\Pi:\Pi\in\mathcal{N}\}$ constitutes a family continuous bijections onto their images; furthermore, this family is  continuous in $\Pi$. Therefore, we obtain a continuous map:
\begin{equation*}
\mathcal{N} \ni \Pi \,\longmapsto \, \hau^N\big({\pr}_\Pi (\G_\mathcal{O})\big) \in \R_+.
\end{equation*}
Noticing that ${\pr}_{\R^N\times\{0\}} \equiv {\pr}$ and $\hau^N(\mathcal{O})>0$, we have found a neighbourhood of $N$-planes  of positive $\mu$-measure about $\R^N\times\{0\}$, such that projections of $\G$ in these directions all have positive $\hau^N$-measure. This contradicts the unrectifiability of $Z$; hence the {\em claim} \eqref{claim} follows.

From now on, one assumes that the graph of $v$ is $(\hau^N,N)$-rectifiable.

\smallskip
\noindent
{\bf Step 4. Reduction to  vectorfields mapping between PL-manifolds.}

Let $v \in C^\infty_c(\Omega;\R^N)$ be such that $\G :={\rm graph}_\Omega (v)$ is $(\hau^N,N)$-rectifiable and that ${\rm spt}\,(v) \Subset \Omega_{\sigma}$ for some $\sigma>0$. In light of the statement of Lemma~\ref{lemma: key}, one has the freedom of modifying $v$ on arbitrarily $\hau^N$-small sets. Thus, using the graphical structure of $\G$ and the definition of rectifiablility, we can assume in the sequel that $\G$ is a $C^1$-submanifold embedded in $\R^{2N}$.

Now, the classical Cairn--Whitehead theorem (\cite{c, wh}) implies that $\G$ has an essentially unique  triangulation $\mathcal{K}_\G$. As $\G={\rm graph}_\Omega(v)$ for $v$ smooth, it induces a triangulation $\mathcal{K}_{\sigma}$ of $\Omega_\sigma \supset {\rm spt}\,(v)$ via the projection $\pr: \Omega \times \R^N\map \Omega$. Hence, here and hereafter, we may view $v$ as a map between PL-manifolds, namely
\begin{equation*}
v: |\mathcal{K}_{\sigma}| \subset \R^N \longrightarrow |\mathcal{K}_\G| \subset \R^{2N}.
\end{equation*}
Throughout, we use $|\mathcal{K}|$ to denote the geometrical realisation of triangulation $\mathcal{K}$.

\smallskip
\noindent
{\bf Step 5: Completion of the proof by simplicial approximation.}

Let $\eta>0$ be arbitrary. By the simplicial approximation theorem (\cite{h}, Theorem~2C.1, p.177), one can find by taking successive barycentric subdivisions of $\mathcal{K}_{\Omega_\sigma}$ (not relabelled) a simplicial map $v_3: |\mathcal{K}_{\Omega_{\sigma}}| \map |\mathcal{K}_\G|$ such that 
\begin{equation}\label{v3}
\|v_3 - v\|_{C^0(\R^N;\R^N)} \leq \frac{\eta}{2}.
\end{equation}
Here, as usual, we identify $v_3$ with its extension-by-zero defined on $\R^N$.

On each simplex $\tau$ of the corresponding barycentric subdivision, $v_3$ is equal to the linear combination of its values at the vertices. Since $v_3$ is a simplicial map, it equals to $\na \psi$ for some $\psi:|\mathcal{K}_{\Omega_{\sigma}}|\map\R$ on each $N$-simplex $\tau$. Since $v_3$ is a $C^0$-map, $\psi$ is $C^1$ in the interior of any such $\tau$. Moreover, since $v$ is uniformly continuous on $\overline{\Omega_\sigma}$, by refining the barycentric subdivisions we can make the oscillation of $v_3$ on any such $\tau$ arbitrarily small. By subtracting a constant, one may further assume that
\begin{equation*}
\max\Big\{\|\psi\|_{C^0(\tau)}:\,\tau \text{ is an $N$-simplex of $\mathcal{K}_{\Omega_\sigma}$}\Big\} \leq \frac{\vartheta}{2}. 
\end{equation*}


Finally, let us modify $\psi$ to obtain the desired map $\phi$, which is $C^1$ on the whole domain $\Omega$. Note that $\psi$ is $C^1$ except on the closure of the $(N-1)$-skeleton of $\mathcal{K}_{\Omega_{{\sigma}}}$, which is a null set with respect to the $N$-dimensional Lebesgue measure. Take an open neighbourhood $\mathcal{O}_2$ of the $(N-1)$-skeleton of arbitrarily small measure, {\it e.g.}, $\hau^N(\mathcal{O}_2) \leq {\e \hau^N(\Omega)}$. A standard smoothing argument then yields $\phi \in C^1_c(\Omega)$ such that $\phi\equiv\psi$ on $\Omega_{\sigma}\sim \mathcal{O}_2$ and that $\|\phi\|_{C^0(\Omega)} \leq \vartheta$. Thanks to \eqref{v3}, we can now complete the proof by suitably choosing $\rho$ and setting $K:= \overline{\Omega \sim \mathcal{O}_2}$.  \end{proof}

The above proof is ``soft'': only geometric measure theoretic and topological arguments are involved, but not hard analysis. It relies essentially on the graphical structure of $\G$. With a little additional effort we can also recover the quantitative $L_p$-estimates in \cite{a}, Lemma~7.

\section{Rectifying the Error}
In this section, we explain how to deduce Theorem~\ref{thm: alberti} ({\it i.e.}, Theorem~1 in Alberti \cite{a})  from Lemma~\ref{lemma: key}. Our treatment is adapted from \cite{a}; nevertheless, we shall establish an approximation theorem in the more general setting of functional analysis. By doing so, we emphasise the central r\^{o}le played by the extension property \eqref{ext ppt}, which imposes severe difficulties for generalising Theorem~\ref{thm: alberti} to function spaces with higher regularity; see Remark~\ref{rem: ext ppt} below.

Let $\XX$ and $\ZZ$ be Banach spaces consisting of Borel functions from an open subset $\Omega \subset \R^N$ (that is, $\XX:=$ the completion of Borel functions on $\Omega$ with respect to $\|\bullet\|_\XX$, and similarly for $\ZZ$), let $\YY \leq \ZZ$ be a Banach subspace, and let $T:\YY \map \XX$ be a bounded linear operator. Suppose that there is a uniform constant $C_0>0$ such that
\begin{equation}\label{eq norm}
{C_0}^{-1} \|\phi\|_{\YY} \leq \|\phi\|_{\ZZ} + \|T\phi\|_{\XX} \leq C_0\|\phi\|_{\YY}.
\end{equation}
For a subset $\Omega' \Subset \Omega$, we denote by $\XX(\Omega')$ the $\|\bullet\|_\XX$-completion of Borel functions over $\Omega'$, and similarly for $\YY(\Omega'), \ZZ(\Omega'), \ldots$. Also, suppose that
\begin{equation}\label{top comparison}
\text{the norm topology of $\XX$ is stronger than the pointwise topology on $\Omega$},
\end{equation}
and that $\|\bullet\|_\XX$ satisfies the {\em extension property}:
\begin{align}\label{ext ppt}
&\text{For each compact subset $K\Subset \Omega$,  each $f \in \XX(K)$, and each $s>0$, there exists $\overline{f} \in \XX\equiv \XX(\Omega)$}\nonumber\\
&\text{such that $\overline{f}|K=f$ and $\|\overline{f}\|_{\XX} \leq (1+s)\|f\|_{\XX(K)}$.}
\end{align}
Then, we have:

\begin{theorem}\label{thm: approx scheme}
Let $\XX$, $\YY$, $\ZZ$, $\Omega$, and $T$  be as above (in particular, the hypotheses \eqref{eq norm}\eqref{top comparison}\eqref{ext ppt} hold); let $v \in \XX$. Assume that for any triplet of positive numbers $\{\eta,\e,\vartheta\}$, there are a compact set $K \Subset \Omega$ and an element $\phi \in \YY$ such that $\hau^N(\Omega \sim K) \leq \e \hau^N(\Omega)$, $\|\phi\|_{\ZZ} \leq \vartheta$, and $\|v-T\phi\|_{\XX(K)} \leq \eta$. Then, for arbitrary positive numbers $\{\delta, \kappa\}$, we can find an open set $A \subset \Omega$ and an element $\tilde\phi \in \YY$ such that $\hau^N(A) < \delta \hau^N(\Omega)$, $\|\tilde\phi\|_{\ZZ} \leq \kappa$, and that $v=T\tilde\phi$ on $\Omega \sim A$.
\end{theorem}

Theorem~\ref{thm: approx scheme} states that, starting from a ``rough'' Lusin-type approximation result with three positive parameters $\{\eta,\e, \vartheta\}$, we can obtain a refined Lusin-type result which remains valid for $\eta=0$; meanwhile, restrictions on the weaker norm $\|\bullet\|_\ZZ$ can be retained. 

Assuming this, Alberti's Theorem~\eqref{thm: alberti} can be deduced as an immediate corollary:

\begin{proof}[Proof of Theorem~\ref{thm: alberti}]
Take $\XX = C^0_0(\Omega;\R^N)$, $\YY= C^1_0(\Omega)$, $\ZZ = C^0_0(\Omega)$, and $T=\na:\YY \map \XX$ as in Theorem~\ref{thm: approx scheme}. In this case, the extension property \eqref{ext ppt} follows from Tietze's theorem (which even allows $s=0$ in \eqref{ext ppt}), and assumptions of Theorem~\ref{thm: approx scheme} on the rough approximation with parameters $\{\eta,\e,\varphi\}$ are verified by Theorem~\ref{lemma: key}. \end{proof}

\begin{proof}[Proof of Theorem~\ref{thm: approx scheme}]
	We construct an iteration scheme: $$\Big\{(v_n, K_n, \varphi_n)\Big\}_{n=0,1,2,\ldots} \subset \XX \times \Big\{\text{compact subsets of $\Omega$}\Big\} \times \YY.$$ 
	
	Take $v_0:=v$, $K_0 := \overline{\Omega}$, and $\varphi_0:=0$. 
	
	Assume that $(v_n, K_n, \varphi_n)$ has been constructed; we shall define $(v_{n+1}, K_{n+1}, \varphi_{n+1})$. First, by assumption, we can take a compact set $K_{n+1}$ such that $\hau^N(\Omega \sim K_{n+1}) \leq 2^{-n-1}\delta \hau^N(\Omega)$, and take $\varphi_{n+1} \in \YY$ such that $\|v_n - T\varphi_{n+1}\|_{\XX(K_{n+1})} \leq 16^{-n}\eta$ as well as $\|\varphi_{n+1}\|_\ZZ \leq 2^{-n-1}\kappa$.  Then, set $v_{n+1}:=v_n - T\varphi_{n+1} \in \XX(K_{n+1})$. In view of the extension property \eqref{ext ppt}, we may extend $v_{n+1}$ to an element of $\XX\equiv\XX(\Omega)$ (without relabelling) such that
\begin{equation}\label{vn+1}
\|v_{n+1}\|_{\XX} \leq 8^{-n}\eta.
\end{equation}	
	
	Now, define $\tilde{\phi}$ and $A$ as follows:
	\begin{equation*}
	\begin{cases}
	\tilde{\phi} := \sum_{j=1}^\infty \varphi_j;\\
	A:= \Omega \sim \bigcap_{j=1}^\infty K_j.
	\end{cases}
	\end{equation*}
It is clear that $A$ is open and $\hau^N(A) \leq \sum_{j=1}^\infty\hau^N(\Omega \sim K_j) \leq \sum_{j=1}^\infty 2^{-j}\delta\hau^N(\Omega) = \delta\hau^N(\Omega)$. Also, $\tilde{\phi}$ is a well-defined element of $\ZZ$ with $\|\tilde{\phi}\|_\ZZ \leq \kappa$. On the other hand, by \eqref{vn+1} we know that $\{v_k\}$ is a  Cauchy sequence in $\XX$, hence it converges to a limit, which must be $0$. Since $T\varphi_{n+1}=v_n-v_{n+1}$, it implies that $\{T\varphi_n\}$ converges in $\XX$. Moreover, by \eqref{eq norm} one obtains
\begin{align*}
\sum_{j=1}^N \|\varphi_j\|_\YY &\leq C_0 \sum_{j=1}^N \Big( \|T\varphi_j\|_\XX + \|\varphi_j\|_\ZZ \Big)\\
&\leq C_0\bigg( \|v\|_{\XX} +  \sum_{j=1}^N (8^{-j-1}+8^{-j})\eta + \sum_{j=1}^N2^{-j}\kappa\bigg)\\
& \leq C_0\Big(\|v\|_{\XX} +  \frac{\eta}{2} + \kappa\Big).
\end{align*}
Sending $N \nearrow \infty$, we deduce that $\tilde{\phi} \in \YY$; in fact, $\|\tilde{\phi}\|_\YY \leq C_0(\|v\|_\XX+ \nicefrac{\eta}{2}+\kappa)$. 

Finally, let us prove that $v=T\tilde{\phi}$ on $\Omega \sim A$. Indeed, \begin{align*}
v-T\tilde{\phi} &= v_0-T\varphi_1 - T\varphi_2 - T\varphi_3 - \ldots
\end{align*}
where $v_0-T\varphi_1=v_1$ on $K_1$, $v_0-T\varphi_1-T\varphi_2=v_2$ on $K_1\cap K_2$, and so on. In view of the extension property \eqref{ext ppt}, we have
\begin{equation*}
\|v_j\|_{\XX\big(\bigcap_{i=1}^j K_i\big)}  \leq (1+s)  4^{-j}\eta.
\end{equation*}
Sending $j\nearrow \infty$ yields that $\Big\|\widehat{v-T\tilde{\phi}}\Big\|_\XX=0$, where $\widehat{v-T\tilde{\phi}}$ denotes the extension of $v-T\tilde{\phi}$ from $\Omega \sim A = \bigcap_{j=1}^\infty K_j$ to $\Omega$, whose existence is ensured by \eqref{ext ppt}. But $v-T\tilde{\phi}$ is defined pointwise on $\Omega\sim A$, so it must be zero thereon, thanks to \eqref{top comparison}. The proof is now complete. \end{proof}

\begin{remark}\label{rem: ext ppt}
	The extension property \eqref{ext ppt} is crucial for our proof of Theorem~\ref{thm: approx scheme}. It is unclear whether an analogous result of Alberti's theorem holds for $\XX=C^{k,\alpha}(\Omega;\R^N)$, $\YY = C^{k+1,\alpha}_0(\Omega)$, $\ZZ=C^{k,\alpha}(\Omega)$, and $T = \na$ for any $k \geq 1$, $\alpha \in ]0,1]$.  There, \eqref{ext ppt} amounts to the hypotheses for Whitney extensions of H\"{o}lder/Lipschitz functions, which are not automatically verified as in the case of Tietze extension for continuous functions. This is reminiscent of S. Delladio's works (see, {\it e.g.}, \cite{de2}) on the higher-order rectifiability criteria on certain generalised fibre bundles. 

\end{remark}

\section{Rectifiable $N$-currents that are non-$C^2$-rectifiable: arbitrary $N$}

An important problem in geometric measure theory  concerns the $C^2$-rectifiability of Legendrian currents, which are natural generalisations of graphs of Gauss maps on hypersurfaces with weaker regularity. Pioneered by Anzellotti--Serapioni \cite{as}, studies on the $C^2$-rectifiability problem have been carried out by Delladio \cite{de1, de2} and Fu \cite{fu1, fu2}, among many other researchers. 

It was first observed in \cite{fu2} that the $C^2$-rectifiability problem is closely related to Alberti's  Theorem~\ref{thm: alberti}. In particular, Fu showed (\cite{fu2}, Proposition~1) that Alberti's construction yields a rectifiable current on $\R^3$ which is not $C^2$-rectifiable. For an $N$-dimensional Riemannian manifold $\M$, we say that $E \subset \M$ is a $C^s$-rectifiable set if and only if there are Borel measurable sets $\{E_0, E_1, E_2, \ldots\} \subset \M$ satisfying $E=\bigcup_{j=0}^\infty E_0$, $\hau^N(E_0)=0$, and for each $j=1,2,\ldots$ there exists $f_j \in C^s(\R^N;\M)$ such that $E_j$ = the image of $f_j$. Thus, rectifiable $\equiv$ $C^1$-rectifiable.

Combining ideas from \cite{fu1, fu2} and elements of contact geometry ({\it cf. e.g.} \cite{e}), we can obtain a ``natural'' rectifiable current of {\em arbitrary degree} that is non-$C^2$-rectifiable. 

\begin{theorem}\label{thm: current}
Let $(\M,g)$ be an arbitrary $N$-dimensional closed Riemannian manifold of with positive volume; $N \geq 2$. There is a rectifiable $N$-current $\sss \in \mathscr{R}_N(\J^1\M)$ which is non-$C^2$-rectifiable. Here $\J^1\M := T^*\M \times \R$ is the {\em $1$-jet space} of $\M$.
\end{theorem}

For this purpose, we need a generalisation of Alberti's Theorem~\ref{thm: alberti}: any differential form is ``nearly'' exact. Moonens--Pfeffer (\cite{mp}, Proposition 2.1) proved this on Euclidean spaces; their  arguments readily generalise to manifolds by a partition of unity argument. We write out the details for completeness. Here and throughout, $\mathscr{D}^k$ denotes the space of $C^\infty$-differential $k$-forms, and ${\rm Vol}_h$ is the volume measure induced by  Riemannian metric $h$. 

\begin{lemma}\label{lem: forms}
Let $(\Sigma,h)$ be a closed, smooth Riemannian manifold of dimension $m \geq 2$ with ${\rm Vol}_g(\Sigma)>0$, let $\omega \in \mathscr{D}^k(\Sigma)$ for $k \geq 1$, and let $\e>0$ be arbitrary. There exist an open set $A \subset \Sigma$ and a $C^1$-differential $(k-1)$-form $\gamma$ on $\Sigma$, such that ${\rm Vol}_h(A) \leq \e {\rm Vol}_h(\Sigma)$ and $\omega = d\gamma$ on $\Sigma \sim A$. 
\end{lemma}

\begin{proof}[Proof of Lemma~\ref{lem: forms}]

Let $\{U_i\}_{i=1}^I$ be a finite smooth atlas for $\Sigma$ with coordinate mappings $\psi_i:U_i \map \R^m$, and let $\{\chi_i\}_{i=1}^I$ be a smooth partition of unity subordinate to $\{U_i\}$. Then we get
\begin{equation*}
\pi^{(i)}:=(\psi_i)_\# (\omega \chi_i) \in \mathscr{D}^k(\R^m)\qquad \text{ for each } i \in \{1,2,\ldots, I\},
\end{equation*}
where $(\psi_i)_\#$ denotes the pushforward under $\psi_i$. Write $\{x^1, \ldots, x^m\}$ for the canonical coordinates on $\R^m$ (fixed for all $i$), and $\Lambda(m,k)$ for the set of multi-indices $\lambda=(\lambda_1, \ldots, \lambda_k) \in \mathbb{N}^k$ such that $1 \leq \lambda_1 < \lambda_2 < \ldots < \lambda_k \leq m$. Then, for each $i \in \{1,2,\ldots, I\}$, we can find smooth coefficient functions $b^{(i)}_\lambda$ compactly supported on $\psi_i(U_i) \subset \R^m$ such that
\begin{equation*}
\pi^{(i)}(x) = \sum_{\lambda \in \Lambda(m,k)} b^{(i)}_\lambda(x) dx^{\lambda_1} \wedge \ldots \wedge dx^{\lambda_k}\qquad \text{ for } x \in \psi_i(U_i).
\end{equation*}

Now, invoking Alberti's Theorem~\ref{thm: alberti} and the canonical isomorphism between vectorfields and $1$-forms, for any $\e'>0$ and $i \in \{1,2,\ldots,I\}$ we can find an open set $\Omega^{(i)} \subset \psi_i(U_i)$ (thanks to the finiteness of the indexing set $\Lambda(m,k)$) and a $C^1_c(\psi_i(U_i))$-map $\phi^{(i,\lambda)}$, such that $\hau^m(\Omega^{(i)}) \leq \e'\hau^m(\psi_i(U_i))$ and $b^{(i)}_\lambda dx^{\lambda_1} = d\phi^{(i,\lambda)}$ outside $\Omega^{(i)}$. Therefore, we get
\begin{align*}
\pi^{(i)} = d\bigg( \sum_{\lambda \in \Lambda(m,k)} \phi^{(i,\lambda)}dx^{\lambda_2} \wedge \ldots\wedge dx^{\lambda_k} \bigg)\qquad \text{ on } \psi_i(U_i) \sim \Omega^{(i)}.
\end{align*}

Take $A:= \bigcup_{i=1}^I \psi_i^{-1}(\Omega^{(i)})$: it is an open set in $\Sigma$; also, by suitably choosing $\e'$ depending on $\e$, $h$, and $\{\psi_i\}$,  we can ensure that ${\rm Vol}_h(A) \leq \e {\rm Vol}_h(\Sigma)$. Furthermore, on $\Sigma \sim A$ there holds 
\begin{align*}
\omega &= \sum_{i=1}^I \omega\chi_i\\
&=  \sum_{i=1}^I\psi_i^\#  d\bigg( \sum_{\lambda \in \Lambda(m,k)} \phi^{(i,\lambda)}dx^{\lambda_2} \wedge \ldots\wedge dx^{\lambda_k} \bigg)\\
&= d\bigg\{ \sum_{i=1}^I \psi_i^\#  \bigg( \sum_{\lambda \in \Lambda(m,k)} \phi^{(i,\lambda)}dx^{\lambda_2} \wedge \ldots\wedge dx^{\lambda_k}  \bigg) \bigg\}.
\end{align*}
The argument inside $\{\cdots\}$ in the last line is clearly a $C^1$-differential $k-1$-form on $\Sigma$. The proof is completed by calling it $\gamma$.  \end{proof}

\begin{proof}[Proof of Theorem~\ref{thm: current}]
	It is proved by Fu (\cite{fu1}, Lemma~1.1; also see \cite{fu2}, Proposition~2) that if a rectifiable $N$-current $\TT$ is carried by a $C^2$-rectifiable set, and if $\beta$ is a smooth differentiable form of degree $\leq (N-1)$ such that $\TT \mres \beta =0$, then $\TT \mres d\beta = 0$. To prove Theorem~\ref{thm: current}, we shall construct $\sss \in \mathscr{R}_N(\J^1\M)$ and $\beta \in \mathscr{D}^1(\J^1\M)$ such that $\sss \mres \beta =0$ while $\sss \mres d\beta \neq 0$.

To begin with, let $\alpha \in \mathscr{D}^1(T^*\M)$ be the tautological $1$-form ({\it a.k.a.} the Liouville/canonical $1$-form) on the cotangent bundle $T^*\M$. We endow $T^*\M$ with the Sasaki metric $\gs$ associated to the Riemannian metric $g$ on $\M$. By assumption, ${\rm Vol}_\gs(T^*\M) = V_0 > 0$. 

Then, by Lemma~\ref{lem: forms}, for any $\varpi \in ]0,1[$ there exist an open set $U\Subset T^*\M$ and a $C^1$-function $\phi:T^*\M\map \R$ such that ${\rm Vol}_\gs(U) \geq \varpi V_0 >0$ and $d\phi=\alpha$ on $U$. 
	
Set $\sss := [[{\rm graph}_U \phi]]$, the current obtained by integrating on the graph of $\phi$ over $U$ with respect to the Sasaki metric. Since $U$ is an open set and $\phi$ is $C^1$, $\sss$ is indeed a rectifiable $N$-current carried by the $1$-jet space $\J^1\M$. Moreover, we take  the natural contact form $\beta := dz-\alpha \in \mathscr{D}^1(\J^1\M)$, where $dz$ is the obvious volume $1$-form on the $\R$ factor of $\J^1\M$. 

Clearly, $\sss \mres \beta$ equals $d\phi-\alpha$ restricted to $U$, which is zero by construction. On the other hand, the $n$-fold wedge product $(d\beta)^{\wedge N}\equiv \pm (d\alpha)^{\wedge N}$. But $d\alpha$ is the canonical symplectic form on $T^*\M$, so $(d\alpha)^{\wedge N}$ coincides the volume form on $(T^*\M,\gs)$. Thus, we have $\sss \mres (d\beta)^{\wedge N} =\pm{\rm Vol}_\gs(U) \neq 0$, which implies that $\sss \mres d\beta \neq 0$. The proof is now complete.  \end{proof}

\end{document}